\documentclass[11pt,reqno,letter]{amsart}

\usepackage{amssymb}
\usepackage[dvipdfmx]{graphicx}

\usepackage{placeins}

\usepackage[usenames]{color}
\usepackage{caption}
\usepackage[dvips]{psfrag}
\usepackage{latexsym}
\usepackage{amsmath}
\usepackage{amsthm}
\usepackage{tabularx}
\usepackage{bm}
\usepackage{lscape}
\usepackage[numbers]{natbib}
\usepackage[usenames]{color}
\usepackage{url}
\usepackage{hyperref}
\usepackage{upgreek}
\usepackage{units}
\usepackage{fancyhdr}
\definecolor{OliveGreen}{rgb}{0,0.6,0}

\usepackage{enumitem}

\usepackage[normalem]{ulem}

\theoremstyle{plain}
\newtheorem{theorem}{Theorem}[section]

\newtheorem{proposition}[theorem]{Proposition}

\theoremstyle{definition}
\newtheorem{definition}[theorem]{Definition}

\newtheorem{assumption}[theorem]{Assumption}

\numberwithin{equation}{section}
\numberwithin{theorem}{section}
\numberwithin{table}{section}
\numberwithin{figure}{section}

\DeclareMathOperator*{\argmax}{arg\,max}

\title[Posterior and Likelihood Sensitivity in Bayesian DRO]{Posterior and Likelihood Sensitivity in \\Bayesian Distributionally Robust Optimization}

\date{\today}

\author[Gotoh]{Jun-ya Gotoh\textsuperscript{$\dagger$}}

\author[Kim]{Michael Jong Kim\textsuperscript{$\ddagger$}}

\author[Lim]{Andrew E.B. Lim\textsuperscript{$*$}}

\dedicatory{\textsuperscript{$\dagger$}Department of Data Science for Business Innovation, Chuo University, Tokyo, Japan. Email: jgoto@kc.chuo-u.ac.jp \\ \textsuperscript{$\ddagger$}Sauder School of Business, University of British Columbia, Vancouver, Canada. Email: mike.kim@sauder.ubc.ca \\ \textsuperscript{$*$}Department of Analytics and Operations, Department of Finance, and Institute of Operations Research and Analytics, National University of Singapore, Singapore. Email: andrewlim@nus.edu.sg}

\begin{document}

\begin{abstract}
We introduce the notion of worst-case posterior and worst-case likelihood sensitivity. These measure, respectively, the sensitivity of the expected cost to worst-case perturbations of the posterior distribution and worst-case perturbations of the likelihood of a Bayesian model. Each defines a quantitative measure of robustness. A decision maker concerned about the sensitivity of the out-of-sample expected cost to deviations from her assumptions will want a decision for which both sensitivities are small. We derive posterior and likelihood sensitivities for uncertainty sets defined in terms of deviation measures. Posterior sensitivity vanishes when the posterior variance shrinks to zero, which occurs when parameter uncertainty is eliminated from learning. Parameter learning does not  eliminate likelihood sensitivity. A distributionally robust formulation of a Bayesian optimization problem makes a near-Pareto-optimal tradeoff between performance (expected cost) and robustness (posterior and likelihood sensitivity).

\end{abstract}

\maketitle


\section{Introduction}


Bayesian models provide a natural framework for learning from data. A Bayesian model requires two inputs from the decision maker (DM), a  likelihood function $L^\theta(\cdot) \equiv L(\cdot|\theta)$ and a prior distribution $\rho$ for the uncertain parameter $\theta$. Together they characterize the joint distribution of the uncertain parameter $\theta$ and the random variable $Y$. The prior on $\theta$ is updated using observed samples of $Y$. The resulting posterior   induces the predictive distribution which is used to compute the expected cost under a decision $x$. A DM  concerned about misspecification would like a small expected cost that is insensitive  to deviations from the predictive distribution.

In this paper, we introduce the notions of posterior and likelihood sensitivity. Posterior sensitivity is the sensitivity of the expected cost to worst-case deviations from the posterior while likelihood sensitivity defines an analogous notion for the likelihood function. They can be interpreted as quantitative measures of robustness.  A DM  concerned about the impact of misspecification on the out-of-sample performance  would like a decision with small expected cost and low posterior and likelihood sensitivity. 
Adapting ideas from \cite{gotoh2018robust,gotoh2021calibration,gotoh2026sensitivity} for robust empirical optimization problems, we show that a broad class of distributionally robust Bayesian optimization problems can be approximated by a ``regularized" Bayesian problem where the regularizer is a weighted sum of posterior and likelihood sensitivities. It follows that approximately Pareto optimal tradeoffs between performance (expected cost) and robustness (posterior and likelihood sensitivity) can be obtained by solving a Bayesian Distributionally Robust Optimization (DRO) problem. 

We derive expressions for posterior and likelihood sensitivity, which depend on the uncertainty sets for posterior and likelihood, and show that posterior sensitivity vanishes when the posterior distribution degenerates to a point mass, but not so for the likelihood. In other words, improvements in parameter estimates from updating the prior with data can reduce posterior, but not likelihood, sensitivity. As illustrated in \cite{gotoh2026sensitivity} for empirical optimization, sensitivity measures can guide the selection of uncertainty sets, identify when the ``price of robustness" is high, and guide system redesign to reduce this cost. 

The outline of the paper is as follows. We briefly review the literature in Section \ref{sec:lit}. We introduce the nominal model (the predictive distribution resulting from the posterior and likelihood) in Section \ref{sec:nominal} and worst-case posterior and worst-case likelihood sensitivity in Section \ref{sec:WCS}. Examples where uncertainty sets for the posterior and likelihood are smooth $\phi$-divergence and formed using different $\phi$-divergences are considered. We show in Section \ref{sec:DROtradeoff} that Bayesian DRO problems map out a nearly Pareto-optimal tradeoff between expected cost, posterior sensitivity and likelihood sensitivity. We consider a simple pricing application in Section \ref{sec:experiments}. We adopt the ``constrained" formulation of the Bayesian DRO problem, where alternative posterior and likelihood distributions are determined by a constraint on the deviation from the nominal. A generalization to the penalty formulation of the Bayesian model can be found in the Appendix.

\section{Literature review}
\label{sec:lit}

This paper is related to three streams of work: approximations of DRO as a regularized nominal problem, robustness measures induced by worst-case sensitivity, and Bayesian DRO.

Several papers \cite{bertsimas2018,blanchet2019,Duchi2017variance,esfahani2018data,gao2023distributionally, gao2024wasserstein,kuhn2019wasserstein} show that DRO can be approximated by a regularized empirical optimization problem. \cite{gotoh2026sensitivity} takes this as the starting point, shows that the regularizer can be interpreted as a quantitative measure of robustness (worst-case sensitivity), and explores the implications for robust modeling including uncertainty set selection (family and size) and system design. The present paper extends this perspective to single stage Bayesian problems. 

A Bayesian model consists of a likelihood function and a prior distribution; the predictive distribution (formed from the posterior and likelihood) is the nominal model. Generalizing ideas from \cite{gotoh2026sensitivity}, we introduce worst-case posterior and worst-case likelihood sensitivity so that a Bayesian decision maker can quantify the fragility of her out-of-sample expected cost to the components of the predictive distribution. We show how each sensitivity measure depends on the uncertainty sets for the posterior and likelihood and that a distributionally robust Bayesian problem maps out a tradeoff between performance (expected cost) and robustness (posterior and likelihood sensitivities). Our results show that posterior sensitivity vanishes when there is no parameter uncertainty, but not so for the likelihood. \cite{shapiro2023bayesian} develops Bayesian DRO and, among other results, derive worst-case likelihood sensitivity  with a Kullback--Leibler uncertainty set. 
We complement this by introducing posterior sensitivity, and considering uncertainty sets for deviation measures beyond Kullback--Leibler divergence.

There is a related body of work in Bayesian statistics that studies the sensitivity of posterior-based inferences  to perturbations of the (subjective) prior and  likelihood (see for example \cite{berger1986Bayesian,gustafson1985Bayesian,Lavine1991Bayesian}). A key difference is the focus on decision making and its role in our model (i.e., we explicitly include the objective function); in particular, how decisions can be chosen when posterior or likelihood sensitivity is large. By showing that the DRO objective can be decomposed into  expected cost plue a weighted sum of posterior sensitivity and likelihood sensitivity, we provide a systematic approach to choosing decisions that balance  performance and sensitivity to misspecification.



\section{Nominal model}
\label{sec:nominal}

Let $(\theta, Y)$ be random variables with a joint distribution ${\mathbb P}({\rm d}\theta, {\rm d}y)$. We denote the marginal distribution ${\mathbb P}({\rm d}\theta) = \rho({\rm d}\theta)$ 
and the conditional distribution  ${\mathbb P}({\rm d}y|\theta) = L^\theta({\rm d}y)$ so
\begin{eqnarray*}
{\mathbb P}({\rm d}\theta, {\rm d}y) = \rho({\rm d}\theta)L^\theta({\rm d}y).
\end{eqnarray*}
Bayesian learning models belong to this class: ${\mathcal L} = \{L^{\theta}\vert \theta\in\Theta\}$ is a parameterized family of likelihoods and $\rho$ is the posterior distribution of the uncertain parameter $\theta$.
The posterior is obtained by updating prior $\rho_0(d\theta)$ using data $Y_1, \cdots, Y_m$ and Bayes' rule. The likelihood function $L^\theta$ is a modeling choice, while the posterior depends on the user-specified prior $\rho_0(d\theta)$ and likelihood $L^{\theta}(dy)$. We are concerned about the sensitivity of the expected cost to deviations from ${\mathbb P}({\rm d}\theta, {\rm d}y)$.  

Another important class is mixture models with populations indexed by $\theta\in\{\theta_1, \cdots, \theta_n\}$ and population distribution ${\mathbb P}({\rm d}y|\theta=\theta_i) = L^{\theta_i}({\rm d}y)$. 
For concreteness, we refer to the marginal distribution $\rho(d\theta)$ as the posterior and the conditional distribution $L^\theta({\rm d}y)$ as the likelihood, fully recognizing that it is more general than a Bayesian learning model. 

Let $f$ be a cost function of the variables $(\theta, Y)$. 
The expected cost is
\begin{eqnarray*}
{\mathbb E}_{\mathbb P}[f(\theta, Y)] = {\mathbb E}_\rho\big\{{\mathbb E}_{L^\theta}[f(\theta, Y)\, | \, \theta]\big\}.
\end{eqnarray*}
The expected cost depends on the joint distribution. We are concerned about the sensitivity of the expected cost relative to worst-case deviations from $\mathbb P$. 


\section{Worst-case sensitivity} 
\label{sec:WCS}

We use ideas from \cite{gotoh2026sensitivity} to define worst-case sensitivity with respect to the posterior and the likelihood.

\subsection{Uncertainty in the posterior and likelihood}

We begin by defining deviations from the posterior $\rho({\rm d}\theta)$ and likelihood ${\mathbb P}({\rm d}y | \theta)$. 

Let $\eta$ be a distribution  on $\Theta$  and ${\rm d}( \eta | \rho)$ be a deviation measure of $\eta$ from $\rho$. For every $\theta\in\Theta$ let ${\mathbb Q}^\theta$ be an alternative to $L^\theta$ for the distribution of  $Y$ and ${\rm d}({\mathbb Q}^\theta| L^\theta)$ a measure of deviation of ${\mathbb Q}^\theta$ from $L^\theta$. It will occasionally be convenient to write ${\mathcal L} = \{ L^\theta| \theta\in\Theta\}$ when we talk about the entire family of distributions.


Let
\begin{eqnarray}
 \label{eq:uncertainty-set}
{\mathcal P}(\varepsilon) & = & \big\{ \eta\big\vert {\rm d}(\eta|\rho)\leq \varepsilon\big\},\; \\  [5pt]
{\mathcal L}^\theta(\delta) & =& \big\{{\mathbb Q}\big\vert {\rm d}({\mathbb Q}|L^\theta)\leq\delta\big\} \nonumber
\end{eqnarray}
be uncertainty sets for the posterior and likelihood distributions. We use the same notation ${\rm d}(\cdot|\cdot)$ for the deviation measure in each uncertainty set. However, different deviation measures can be used for the posterior and likelihood (see Section \ref{sec:mix}).

Consider the worst-case expected cost
\begin{eqnarray}
V(\epsilon, \delta) 
= \max_{\eta \in {\mathcal P}(\varepsilon)} {\mathbb E}_\eta \Big\{\max_{{\mathbb Q}^\theta \in {\mathcal L}^\theta(\delta)} {\mathbb E}_{Q^\theta} [f(\theta, Y)] \Big\}.
\label{eq:wc_Bayes}
\end{eqnarray}
The inner expectation is conditional on $\theta$; we consider worst-case deviations ${\mathbb Q}^\theta$ from the nominal likelihood $L^\theta$ for each $\theta$. The outer expectation considers worst-case deviations $\eta$ from the posterior $\rho$ so changes the weights on $\theta$. 

\subsection{Regularization}

Consider first the inner problem. Since the worst-case expected cost is increasing in the size of the uncertainty set $\delta$, we can write for every $\theta\in\Theta$
\begin{eqnarray}
\Psi(\theta) = \max_{{\mathbb Q}^\theta \in {\mathcal L}^\theta(\delta)} {\mathbb E}_{{\mathbb Q}^\theta} [f(\theta, Y)] 
   =    {\mathbb E}_{L^\theta}[f(\theta, Y)]  + {\mathcal A}_{L^\theta}\big(\delta; f(\theta, Y)\big)
   \label{eq:inner}
   \end{eqnarray}
where   ${\mathcal A}_{L^\theta}\big(\delta; f(\theta, Y)\big)$, the ambiguity cost, is non-negative and non-decreasing in $\delta$. 

Similarly, we can write the outer problem
\begin{eqnarray}
\max_{\eta \in {\mathcal P}(\varepsilon)} {\mathbb E}_\eta \big\{\Psi(\theta) \big\} & = & {\mathbb E}_\rho \big\{\Psi(\theta) \big\} + {\mathcal A}_{\rho}\big(\varepsilon; \Psi(\theta)\big)
\label{eq:outer}
\end{eqnarray}
where the ambiguity cost ${\mathcal A}_{\rho}\big(\varepsilon; \Psi(\theta)\big)$ is again non-negative and increasing in $\varepsilon$.

Substituting \eqref{eq:inner} into \eqref{eq:outer} it  follows that the worst-case expected cost \ref{eq:wc_Bayes} can be written
\begin{eqnarray*}
V(\epsilon, \delta)  = {\mathbb E}_\rho \Big\{{\mathbb E}_{L^\theta}[f(\theta, Y)]  \Big\} + {\mathcal A}_{\rho,{\mathcal L}}\big(\varepsilon, \delta; f(\theta, Y)\big)
\end{eqnarray*}
with ambiguity cost
\begin{eqnarray*}
{\mathcal A}_{\rho,{\mathcal L}}\big(\varepsilon, \delta; f(\theta, Y)\big) := {\mathbb E}_\rho \Big\{ {\mathcal A}_{L^\theta}\big(\delta; f(\theta, Y)\big) \Big\} + {\mathcal A}_{\rho}\Big(\varepsilon; {\mathbb E}_{L^\theta}[f(\theta, Y)]  + {\mathcal A}_{L^\theta}\big(\delta; f(\theta, Y)\big)\Big).
\end{eqnarray*}

We make the following assumptions about the ambiguity costs. The functions $g(\varepsilon)$ and $g(\delta)$ for the two ambiguity costs need not be the same.
\begin{assumption} \label{ass:Bayes}
The ambiguity costs ${\mathcal A}_{L^\theta}\big(\delta; f(\theta, Y)\big)$  and ${\mathcal A}_{\rho}\big(\varepsilon; \Psi(\theta)\big)$ are such that 
\begin{enumerate}
\item there is a non-decreasing function $g(\delta)$ such that $g(\delta)\rightarrow 0$ when $\delta\rightarrow 0$ and 
\begin{eqnarray*}
{\mathcal A}_{L^\theta}\big(\delta; f(\theta, Y)\big) 
   =  g(\delta) {\mathcal S}_{L^\theta}(f(\theta, Y))+ o(g(\delta))
\end{eqnarray*}
for every $\theta\in{\Theta}$;
\item there is a non-decreasing function $g(\varepsilon)$ such that $g(\varepsilon)\rightarrow 0$ when $\varepsilon\rightarrow 0$ and 
\begin{eqnarray*}
{\mathcal A}_{\rho}\big(\varepsilon; \Psi(\theta)\big) 
=  
g(\varepsilon) {\mathcal S}_\rho(\Psi(\theta)) + o(g(\varepsilon)).
\end{eqnarray*}
\item  ${\mathcal S}_{\rho}(\Psi)$ is continuous in $\Psi$. 
\end{enumerate}
\end{assumption}
Intuitively, $\Psi(\theta)$ is the worst-case expected cost when  the uncertainty set  is $\{{\mathcal L}^\theta(\delta)\}$ and
\begin{eqnarray*}
{\mathcal A}_{L^\theta}\big(\delta; f(\theta, Y)\big)= \max_{{\mathbb Q}^\theta \in {\mathcal L}^\theta(\delta)} {\mathbb E}_{{\mathbb Q}^\theta} [f(\theta, Y)] 
   -   {\mathbb E}_{L^\theta}[f(\theta, Y)]
   \end{eqnarray*} 
is the  increase in expected cost relative to the nominal. When Assumption \ref{ass:Bayes} holds,  ${\mathcal S}_{L^\theta}(f(\theta, Y))$  can be interpreted as the  sensitivity of the expected cost under worst-case perturbations; it depends on the deviation measure ${\rm d}({\mathbb Q} | L^\theta)$ that defines the uncertainty set. For example, ${\mathcal S}_{L^\theta}(f(\theta, Y))$ is the standard deviation of the cost $\sigma_{L^\theta} (f(\theta, Y))$ when ${\rm d}({\mathbb Q}|L^\theta)$ is smooth $\phi$-divergence. Sensitivity measures induced by other measures of deviation can be found in \cite{gotoh2026sensitivity} and references cited therein (see also Section \ref{sec:mix}).
   
Likewise, ${\mathbb E}_\rho \big\{\Psi(\theta) \big\}$ is the expected value of $\Psi(\theta)$ when $\theta$ has distribution $\rho$ and
${\mathcal S}_{\rho}(\Psi)$ is the sensitivity of the expected cost ${\mathbb E}_\rho \big\{\Psi(\theta) \big\}$ under worst-case perturbations of $\rho$. 
   
Uncertainty in posterior and likelihood contribute to increases in the total expected cost
   \begin{eqnarray*}
   {\mathcal A}_{\rho,{\mathcal L}}\big(\varepsilon, \delta; f(\theta, Y)\big) = \max_{\eta \in {\mathcal P}(\varepsilon)} {\mathbb E}_\eta \Big\{\max_{{\mathbb Q}^\theta \in {\mathcal L}^\theta(\delta)} {\mathbb E}_{Q^\theta} [f(\theta, Y)] \Big\} - {\mathbb E}_\rho \Big\{{\mathbb E}_{L^\theta}[f(\theta, Y)]  \Big\}.
   \end{eqnarray*}
Our goal is to understand how the sensitivity of the total cost depends on the posterior and likelihood.

The following result gives an expansion of the value function when $\varepsilon$ and $\delta$ are small. 
\begin{theorem} \label{prop:Bayesian expansion}
Suppose that Assumption \ref{ass:Bayes} holds. Then
\begin{eqnarray*}
V(\epsilon, \delta) & = & 
{\mathbb E}_\rho\Big\{{\mathbb E}_{L^\theta}[f(\theta, Y)] \Big\} + g(\delta){\mathbb E}_\rho\Big\{{\mathcal S}_{L^\theta}(f(\theta, Y))\Big\}  +  g(\varepsilon) {\mathcal S}_\rho\Big({\mathbb E}_{L^\theta}(f(\theta, Y))\Big) \\ [5pt] 
& & + o(g(\delta)) + o(g({\varepsilon})).
\end{eqnarray*}
\end{theorem}

\begin{proof}
It follows from \eqref{eq:inner} and Assumption \ref{ass:Bayes} that
\begin{eqnarray*}
\Psi(\theta)  &= &   {\mathbb E}_{L^\theta}[f(\theta, Y)]  + {\mathcal A}_{L^\theta}\big(\delta; f(\theta, Y)\big) \\ & = & {\mathbb E}_{L^\theta}[f(\theta, Y)]  + g(\delta) {\mathcal S}_{L^\theta}(f(\theta, Y))+ o(g(\delta)).
\end{eqnarray*}
From \eqref{eq:outer} and Assumption \ref{ass:Bayes} we have
\begin{align*}
V(\varepsilon, \delta) & =  \max_{\eta \in {\mathcal P}(\varepsilon)} {\mathbb E}_\eta \big\{\Psi(\theta) \big\} \\[8pt] 
& = \max_{\eta \in {\mathcal P}(\varepsilon)} {\mathbb E}_\eta\Big\{{\mathbb E}_{L^\theta}[f(\theta, Y)]  + g(\delta) {\mathcal S}_{L^\theta}(f(\theta, Y))+ o(g(\delta))\Big\} \\[8pt]
& = {\mathbb E}_\rho\Big\{{\mathbb E}_{L^\theta}[f(\theta, Y)]+ g(\delta) {\mathcal S}_{L^\theta}(f(\theta, Y))+ o(g(\delta)) \Big\} \\[8pt]
& \qquad + g(\varepsilon){\mathcal S}_\rho\Big({\mathbb E}_{L^\theta}[f(\theta, Y)]  + g(\delta) {\mathcal S}_{L^\theta}(f(\theta, Y))+ o(g(\delta))\Big) + o(g(\varepsilon)) \\[8pt]
& = {\mathbb E}_\rho\Big\{{\mathbb E}_{L^\theta}[f(\theta, Y)] \Big\} + g(\delta){\mathbb E}_\rho\Big\{{\mathcal S}_{L^\theta}(f(\theta, Y))\Big\}  +  g(\varepsilon) {\mathcal S}_\rho\Big({\mathbb E}_{L^\theta}(f(\theta, Y))\Big) \\[8pt]
& \qquad + g(\varepsilon) \left\{{\mathcal S}_\rho\Big({\mathbb E}_{L^\theta}[f(\theta, Y)]  + g(\delta) {\mathcal S}_{L^\theta}(f(\theta, Y)) + o(g(\delta)) \Big) - {\mathcal S}_\rho\Big({\mathbb E}_{L^\theta}(f(\theta, Y))\Big)\right\}  
\\[8pt]
& 
\qquad + o(g(\delta)) + o(g({\varepsilon})) \\[8pt]
& = {\mathbb E}_\rho\Big\{{\mathbb E}_{L^\theta}[f(\theta, Y)] \Big\} + g(\delta){\mathbb E}_\rho\Big\{{\mathcal S}_{L^\theta}(f(\theta, Y))\Big\}  +  g(\varepsilon) {\mathcal S}_\rho\Big({\mathbb E}_{L^\theta}(f(\theta, Y))\Big) \\[8pt] 
& \qquad + o(g(\delta)) + o(g({\varepsilon}))
\end{align*}
where the last equality follows from the continuity of ${\mathcal S}_\rho(\Psi)$.
\end{proof}

\subsection{Worst-case sensitivity}

Given likelihood function $L^\theta(y)$ ($\theta\in\Theta$) and posterior $\rho$, 
worst-case posterior sensitivity is the sensitivity of the expected cost under worst-case deviations from the posterior defined by the uncertainty set ${\mathcal P}(\varepsilon)$:
\begin{eqnarray*}
{\mathcal S}_\rho(f)  =  \lim_{\epsilon\rightarrow 0}\frac{V(\epsilon, 0)- {\mathbb E}_\rho\Big\{{\mathbb E}_{L^\theta}[f(\theta, Y)]\Big\}}{g({\epsilon})}
=
{\mathcal S}_\rho\Big\{{\mathbb E}_{L^\theta}[f(\theta, Y)]\Big\}.
\end{eqnarray*}

Given the nominal posterior $\rho$, likelihood  $L^\theta(y)$, and uncertainty set ${\mathcal L}^\theta(\delta)$ for deviations from the likelihood, worst-case likelihood sensitivity is the sensitivity of the expected cost with respect to worst-case deviations from the nominal likelihood
\begin{eqnarray*}
{\mathcal S}_{\mathcal L}(f)  =  \lim_{\delta\rightarrow 0}\frac{V(0, \delta)- {\mathbb E}_\rho\Big\{{\mathbb E}_{L^\theta}[f(\theta, Y)]\Big\}}{g(\delta)}
 =  {\mathbb E}_\rho \Big\{{\mathcal S}_{L^\theta}(f(\theta, Y))\Big\}.
\end{eqnarray*}
Note that posterior (likelihood) sensitivity ${\mathcal S}_\rho$  depends on the deviation measure ${\rm d}(\eta|\rho)$  that defines the uncertainty set ${\mathcal P}(\varepsilon)$, and likelihood sensitivity  ${\mathcal S}_{L^\theta}$ on the deviation measure  ${\rm d}({\mathbb Q}|L^\theta)$ that controls deviations from the nominal likelihood.

In general, ${\mathcal S}_\rho$ is a generalized measure of deviation \cite{gotoh2026sensitivity,rockafellar2006generalized} so posterior sensitivity is zero when $\rho$ is a point mass (i.e., there is no uncertainty about $\theta$) or when $g(\theta) = {\mathbb E}_{L^\theta}[f(\theta, Y)]$ is constant on the support of $\rho$. 
Likelihood sensitivity is zero only when, for $\rho$-almost every $\theta$, the sensitivity of the cost under $L^\theta$ is zero. For the deviation measures considered here, this occurs when $f(\theta, y)$ is constant in $y$ on the support of $L^\theta$.

The following example considers posterior and likelihood sensitivity when deviation measures are smooth $\phi$-divergence.

\subsection{Smooth $\phi$-divergence}
 \label{eg:phi_Bayes}

Suppose that $\eta$ is absolutely continuous with respect to $\rho$ and  ${\mathbb Q}^\theta$ is absolutely continuous with respect to $L^\theta$ for every $\theta\in\Theta$. Let $\frac{{\rm d} \eta}{{\rm d}\rho}$ denote the Radon--Nikodym derivative of $\eta$ with respect to $\rho$, and for every $\theta\in\Theta$, $\frac{{\rm d}Q^\theta}{{\rm d}L^\theta}$ be the Radon--Nikodym derivative of $Q^{\theta}$ with respect to $L^\theta$.

For the likelihood function,  $\phi$-divergence of ${\mathbb Q}^\theta$ with respect to $L^\theta$, for every $\theta\in{\Theta}$, is defined by
\begin{eqnarray*}
 {\rm d}({\mathbb Q}^\theta|L^\theta)= {\mathbb E}_{L^\theta} \left[\phi \left(\frac{{\rm d}{\mathbb Q}^\theta}{{\rm d} L^\theta}\right) \right].
\end{eqnarray*}
In the case of smooth $\phi$-divergence, Assumption \ref{ass:Bayes} holds and the inner problem \cite{gotoh2026sensitivity} 
\begin{eqnarray*}
\max_{{\mathbb Q}^\theta \in {\mathcal L}^\theta_{\phi}(\delta)} {\mathbb E}_{Q^\theta} [f(\theta, Y)] 
=  {\mathbb E}_{L^\theta}[f(\theta, Y)]  + \sqrt{\frac{2\delta}{\phi''(1)}} {\sigma}_{L^\theta}(f(\theta, Y))+ o(\sqrt{\delta}).
\end{eqnarray*}
For the outer problem we measure the deviation of $\eta$ from $\rho$ using $\phi$-divergence:
\begin{eqnarray*}
{\rm d}( \eta | \rho) & = & {\mathbb E}_\rho \left[\phi\left(\frac{{\rm d}\eta}{{\rm d}\rho}\right)\right]
\end{eqnarray*}
and it follows that
\begin{eqnarray*}
\max_{\eta \in {\mathcal P}_\phi(\varepsilon)} {\mathbb E}_\eta[\Psi(\theta)] 
= {\mathbb E}_\rho[\Psi(\theta)] + \sqrt{\frac{2\varepsilon}{\phi''(1)}}\sigma_\rho(\Psi(\theta)) + o(\sqrt{\varepsilon}).
\end{eqnarray*}
In particular, $g(\delta) = \sqrt{\delta}$ and 
\begin{eqnarray*} 
{\mathcal S}_{L^\theta}(f) =  \sqrt{\frac{2}{\phi''(1)}}  {\sigma}_{L^\theta}(f(\theta, Y))
\end{eqnarray*}  
for the inner problem, where ${\sigma}_{L^\theta}(f(\theta, Y)) = \sqrt{ {\mathbb V}_{L^\theta}(f(\theta, Y))}$ is the standard deviation of the cost under the distribution $L^\theta$. For the outer problem,
$g(\varepsilon) = \sqrt{\varepsilon}$ and
\begin{eqnarray*}
{\mathcal S}_\rho(\Psi(\theta))  =\sqrt{\frac{2}{\phi''(1)}}\sigma_\rho(\Psi(\theta)).
\end{eqnarray*}
Note that the regularizer in the expansion of the worst-case expected costs are given by the standard deviation because we are using smooth $\phi$-divergence as the deviation measure. Other uncertainty sets give different regularizers \cite{gotoh2026sensitivity}.

It now  follows from Theorem \ref{prop:Bayesian expansion} that the worst-case expected cost 
\begin{eqnarray*}
V(\epsilon, \delta) 
& = &  {\mathbb E}_\rho\Big\{{\mathbb E}_{L^\theta}[f(\theta, Y)]\Big\}+ \sqrt{\frac{2\delta}{\phi''(1)}}{\mathbb E}_\rho \Big({\sigma}_{L^\theta}(f(\theta, Y))\Big) + \sqrt{\frac{2\varepsilon}{\phi''(1)}}
{\sigma}_\rho\Big({\mathbb E}_{L^\theta}[f(\theta, Y)]\Big) \nonumber \\
& & + o(\sqrt{\epsilon}) + o(\sqrt{\delta}).
\end{eqnarray*}
Worst-case posterior  sensitivity is
\begin{eqnarray}
{\mathcal S}_\rho(f) 
= \sqrt{\frac{2}{\phi''(1)}}{\sigma}_\rho\Big({\mathbb E}_{L^\theta}[f(\theta, Y)]\Big)
\label{eq:prior-sensitivity-const}
\end{eqnarray}
and worst-case likelihood sensitivity is 
\begin{eqnarray}
{\mathcal S}_{\mathcal L}(f) 
= \sqrt{\frac{2}{\phi''(1)}} {\mathbb E}_\rho \Big\{{\sigma}_{L^\theta}(f(\theta, Y))\Big\}.
\label{eq:likelihood-sensitivity-const}
\end{eqnarray}
Both sensitivity measures depend on the posterior and objective function, but quite differently. Posterior sensitivity  is the standard deviation of the conditional expectation $g(\theta) = {\mathbb E}_{L^\theta}[f(\theta, Y)]$ and will be small if $g(\theta)$ is ``flat" on the support of $\rho$, even if the standard deviation of $\theta$ is large. 
It vanishes when the posterior variance shrinks (e.g.) because of learning. Likelihood sensitivity depends on the variance of the cost given $\theta$. It typically will not vanish even if $\theta$ is known because there is still uncertainty about the distribution $L^\theta$ of $Y$.

The paper \cite{shapiro2023bayesian} derives worst-case likelihood sensitivity  with a Kullback--Leibler uncertainty set (the inner ``max" in \eqref{eq:wc_Bayes}) but not the posterior. We complement this by introducing posterior sensitivity and by showing how the worst-case objective can be approximated by  expected cost, posterior sensitivity, and likelihood sensitivity.

\subsection{Mix and match}
\label{sec:mix}

As noted in the discussion following \eqref{eq:uncertainty-set}, there is no reason to restrict ourselves to smooth $\phi$-divergence for the deviation measure or to have the same deviation measure for the posterior and likelihood. For example, if instead of smooth $\phi$-divergence, we use 
\begin{eqnarray*}
\phi(z)=\delta_{[1-\varepsilon,\frac{1-(1-\varepsilon)\alpha}{1-\alpha}]}(z)
\end{eqnarray*} 
to define the uncertainty set for the likelihood function, we have \cite{gotoh2026sensitivity}\footnote{Given $\theta$, $\mathrm{CVaR}_{L^\theta,\alpha}(f(\theta, Y))$ is the conditional-value-at-risk of the cost $f(\theta, Y)$ at level $\alpha$ when $Y$ has distribution $L^\theta$: \begin{eqnarray*}\mathrm{CVaR}_{L^\theta,\alpha}(f(\theta, Y)):=\min_\gamma\Big\{\gamma+\frac{1}{1-\alpha}\mathbb{E}_{L^\theta}\Big(\max\big\{f(\theta, Y)-\gamma,0\big\}\Big)\Big\}\end{eqnarray*}}
\begin{eqnarray*} 
{\mathcal S}_{L^\theta}(f) = \mathrm{CVaR}_{L^\theta,\alpha}(f(\theta, Y))-\mathbb{E}_{L^\theta} (f(\theta, Y)).
\end{eqnarray*} 
(WCS for other deviation measures are also derived in \cite{gotoh2026sensitivity}). We choose this uncertainty set and deviation measure if we are concerned about misspecification of the right tail of the cost distribution. 
If smooth $\phi$-divergence is used for the posterior and this alternative is used for the likelihood, we have
\begin{align*}
{\mathcal S}_{\rho}(f) &= \sqrt{\frac{2}{\phi''(1)}}{\sigma}_\rho\Big({\mathbb E}_{L^\theta}[f(\theta, Y)]\Big), \\[5pt]
{\mathcal S}_{\mathcal L}(f) & = {\mathbb E}_\rho \Big\{\mathrm{CVaR}_{L^\theta, \alpha}(f(\theta, Y))-\mathbb{E}_{L^\theta} (f(\theta, Y))\Big\}.
\end{align*}

\section{Distributionally Robust Optimization}
\label{sec:DRO}

\subsection{Performance--sensitivity tradeoffs in DRO}
\label{sec:DROtradeoff}

Theorem \ref{prop:Bayesian expansion} shows that the  worst-case expected cost can be locally expanded as
\begin{eqnarray}
\lefteqn{\min_x \max_{\eta \in {\mathcal P}(\varepsilon)} {\mathbb E}_\eta \Big\{\max_{{\mathbb Q}^\theta \in {\mathcal L}^\theta(\delta)} {\mathbb E}_{Q^\theta} [f(x, \theta, Y)] \Big\}}
\label{eq:DRO-regularized}
\\ [5pt]
& = & \min_x {\mathbb E}_\rho\Big\{{\mathbb E}_{L^\theta}[f(x, \theta, Y)]\Big\}+ g(\delta) {\mathbb E}_\rho \Big({\mathcal S}_{L^\theta}(f(x, \theta, Y))\Big) + g(\varepsilon)  {\mathcal S}_\rho\Big({\mathbb E}_{L^\theta}[f(x, \theta, Y)]\Big) 
\nonumber  \\ [5pt] 
& & + o(g({\epsilon})) + o(g({\delta})).
\nonumber
\end{eqnarray}
This representation is related to results showing the relationship between DRO and a regularized nominal problem. For empirical optimization problems, \cite{gotoh2026sensitivity} shows that the regularization term is not just a mathematical object that connects the nominal and worst-case problems, but  has the physical interpretation as worst-case sensitivity. \eqref{eq:DRO-regularized} shows an analogous result for a Bayesian DRO problem, where the regularizer is now a weighted sum of posterior  sensitivity ${\mathcal S}_\rho({\mathbb E}_{L^\theta}[f(x, \theta, Y)])$   and likelihood sensitivity ${\mathbb E}_\rho ({\mathcal S}_{L^\theta}(f(x, \theta, Y)))$, and varying the uncertainty set sizes maps out a tradeoff between performance (expected cost) and robustness (posterior and likelihood sensitivity). 



\section{Experiment}
\label{sec:experiments}

Demand $D(p)$ is normal with standard deviation $\sigma$ and expected value ${\mathbb E}[D(p)] = a + b p$ which depends on the price $p$. We assume that $\sigma$ is known and constant whereas $a$ and $b$ are constant but unknown to the DM. At the time of the decision, the decision maker has a posterior on $(a, b)$ which we assume to be normal. (This could have been obtained by updating an initial prior using data and Bayes' rule.) The DM's revenue when price is $p$ and observed demand is $D(p)$ is $p\cdot D(p)$. The expected revenue under the posterior is ${\mathbb E}_\rho [p \, D(p)] = p \cdot ({\mathbb E}_\rho[a]  + {\mathbb E}_\rho[b]\cdot p)$; the nominal DM chooses price to maximize this quantity. 

We assume a modified $\chi^2$--uncertainty set for the posterior and likelihood, i.e., $\phi(z) = \frac{1}{2}(z-1)^2$. It follows that posterior and likelihood sensitivities are given by
\begin{align}
\label{eq:pricing-senitivities}
{\mathcal S}_\rho[p\, D(p)] & = \sqrt{2} p \, \sigma_\rho [a + b p], \\
{\mathcal S}_{\mathcal L}[p\, D(p)] & = \sqrt{2}\sigma p. \nonumber
\end{align}
For this experiment, under the baseline posterior, $a$ has mean $5$ and standard deviation $0.75$, and $b$ has mean $-0.5$ and standard deviation $2.5$; we assume $a$ and $b$ are independent under the posterior. We set $\sigma = 3$.  For the nominal problem the optimal price is $p =\$5$; the optimal expected reward is $\$12.50$.

Figure \ref{fig:frontier} shows reward--posterior sensitivity and reward--likelihood sensitivity frontiers when the posterior variance for $a$ and $b$ is 0, $1/4$ of the baseline, the baseline, and double the baseline. This simulates the effect of learning, which reduces the posterior variance, on posterior and likelihood sensitivity. We see in the first plot that posterior sensitivity is increasing in the posterior variance, quantifying the sensitivity of the expected reward to the posterior. Note that posterior sensitivity is $0$ when there is no parameter uncertainty. The DRO solution makes a tradeoff between expected reward and worst-case posterior sensitivity. For this particular example, DRO is quite effective when the posterior variance is large ($2\times$baseline) as substantial reduction in the posterior sensitivity from its value under the nominal model can be achieved with minimal impact on expected cost. It is less effective at reducing posterior sensitivity when the posterior variance is small ($0.25 \times$baseline). However, the frontier shows that this is not really necessary as the posterior sensitivity is already relatively small. 

\begin{figure}[!htbp]
\begin{center}
\includegraphics[scale=0.26]{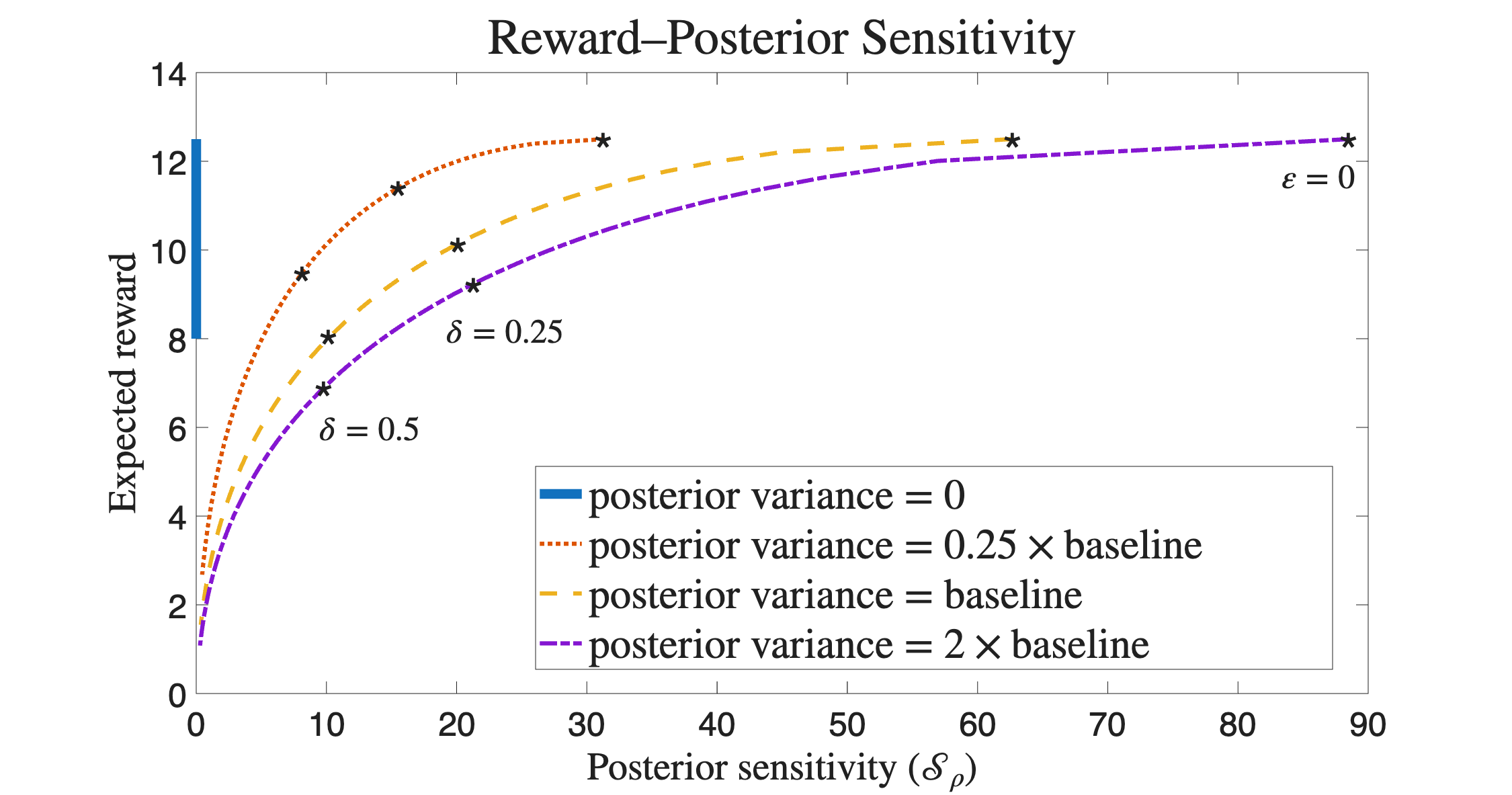}  \\ [5pt]
\includegraphics[scale=0.26]{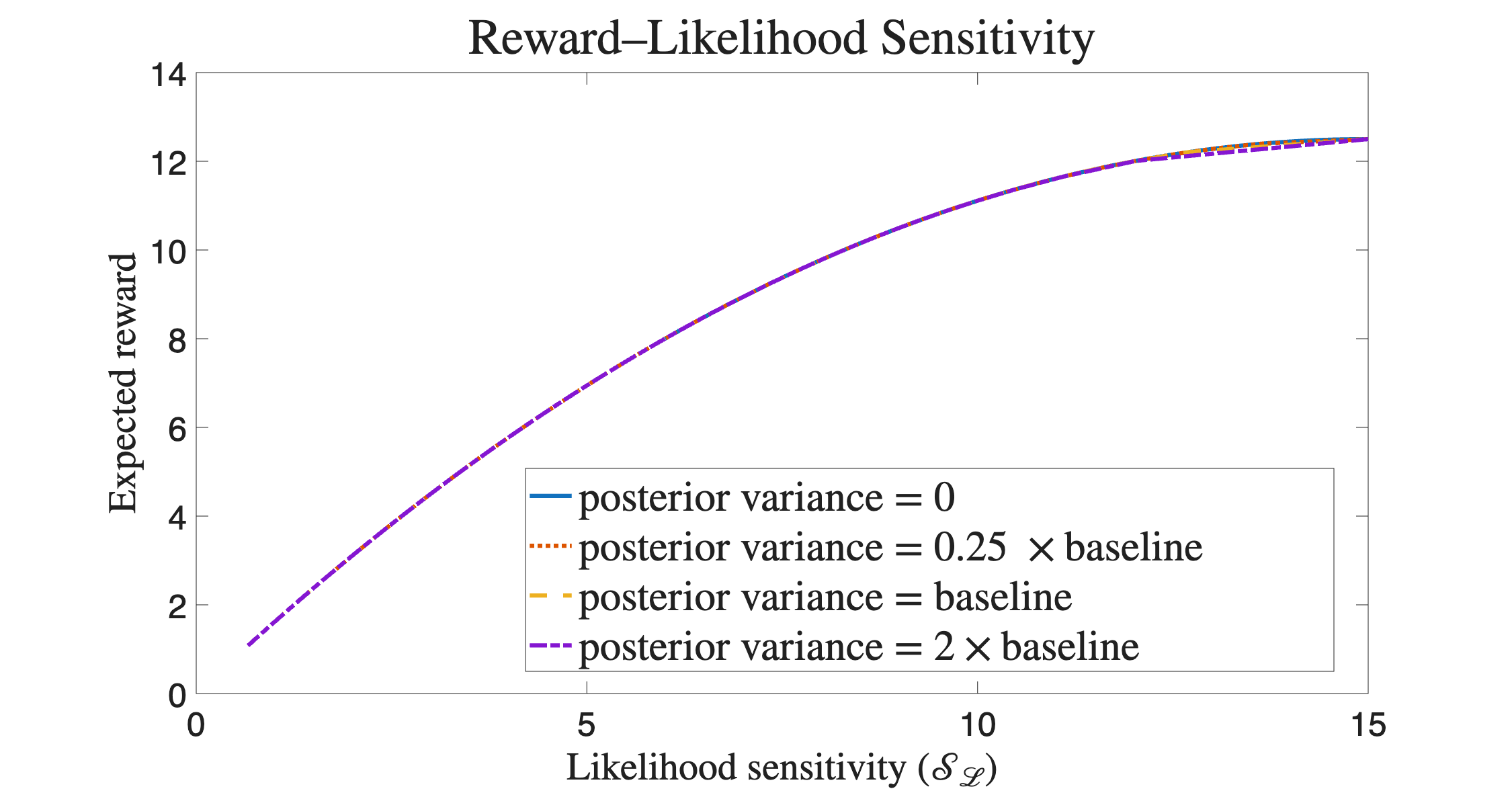} 
\end{center}
\caption{Reward--posterior sensitivity frontiers (upper), and reward--likelihood sensitivity frontiers (lower), for several posterior variances. Uncertainty sets for posterior and likelihood uncertainty are equal in size (i.e., $\delta = \varepsilon$); points  corresponding to $\varepsilon=0, 0.25$ and $0.5$ are indicated in the top plot. The expected reward under the nominal ($\varepsilon=\delta=0$) is $12.5$. From  the upper plot, posterior--sensitivity of the nominal solution is increasing in the  posterior variance; it vanishes when posterior variance is $0$. For this example, likelihood sensitivity is independent of the posterior variance.  }
\label{fig:frontier}
\end{figure}

From the lower plot, we see that the mean--likelihood sensitivity frontiers lie on top of one another (the small deviations are from numerical errors). This is not a general property but a quirk of this example; it would not be the case if $\sigma$ was also uncertain\footnote{
It is easy to see that
\begin{eqnarray*}
{\mathbb E}_\rho[p\,D(p)] = \frac{{\mathcal S}_{\mathcal L}}{\sigma\sqrt{2}}{\mathbb E}_\rho[a] + \left(\frac{{\mathcal S}_{\mathcal L}}{\sigma\sqrt{2}}\right)^2 {\mathbb E}_\rho[b] 
\end{eqnarray*}
so the expected reward, given likelihood sensitivity ${\mathcal S}_{\mathcal L}$, does not depend on the posterior variance. It follows that the reward--likelihood sensitivity frontiers do not depend on the posterior variance. It can also be shown that the reward--posterior sensitivity frontier does not depend on the likelihood variance $\sigma^2$ (again, a quirk of the problem).}. The main message, however, is that we  tradeoff between expected reward, posterior sensitivity, and likelihood sensitivity as the size of the uncertainty set $\varepsilon=\delta$ changes, with DRO prices being mapped  in Figure  \ref{fig:DRO_price}.

\begin{figure}[!htbp]
\begin{center}
\includegraphics[scale=0.26]{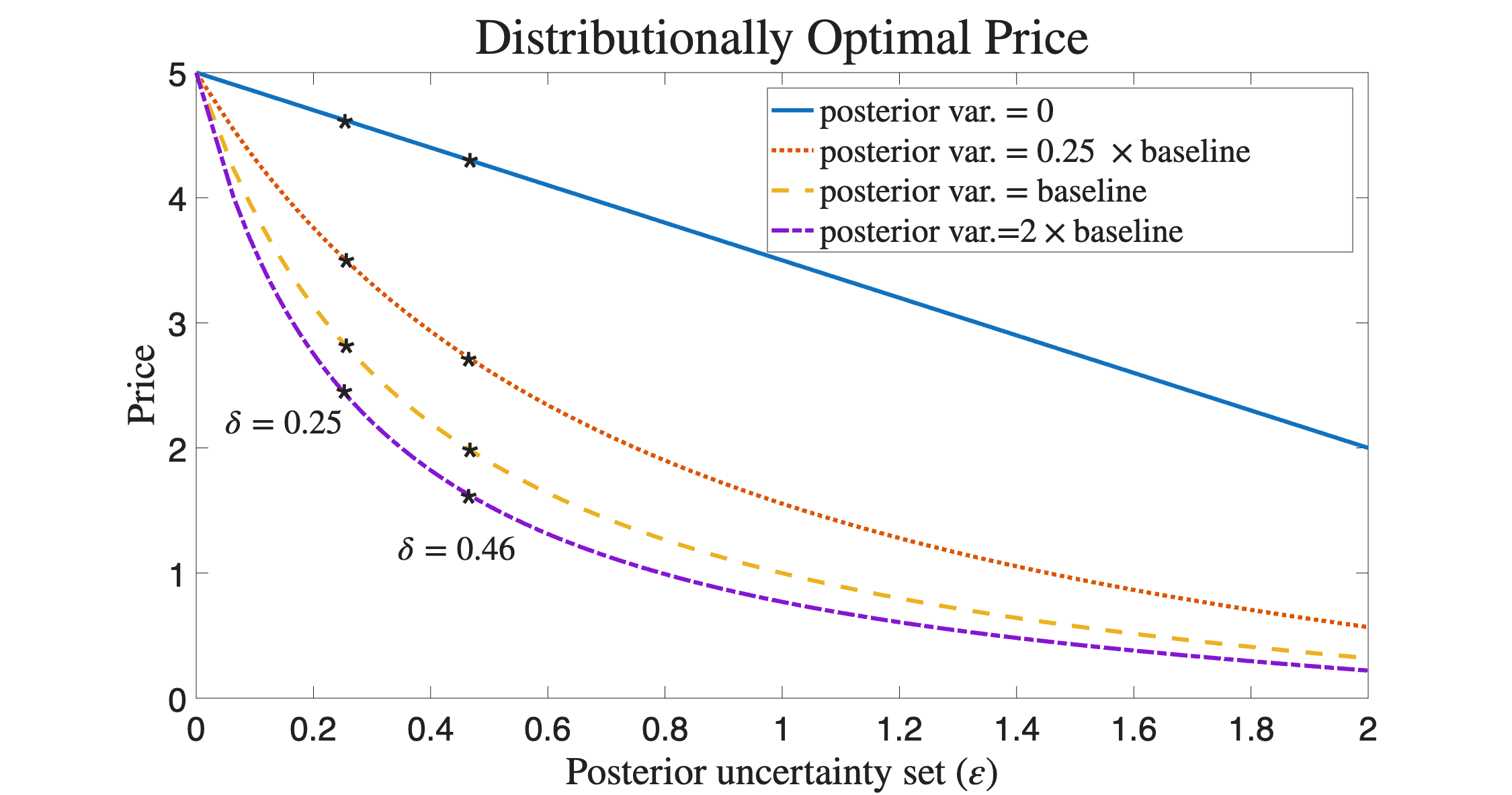}  
\end{center}
\caption{Optimal robust price as a function of the size of the uncertainty set ($\varepsilon=\delta$) for different posterior variance.
}
\label{fig:DRO_price}
\end{figure}

We can see from \eqref{eq:pricing-senitivities} that the likelihood sensitivity does not vanish when the posterior variance is $0$, so long as $p\neq 0$. This is a general property: while more data shrinks the posterior and eliminates  posterior sensitivity (as observed in Figure \ref{fig:frontier}),  uncertainty about the distribution of demand associated with the likelihood model remains unless $\sigma=0$.

 \begin{figure}[!htbp]
\begin{center}
\includegraphics[scale=0.26]{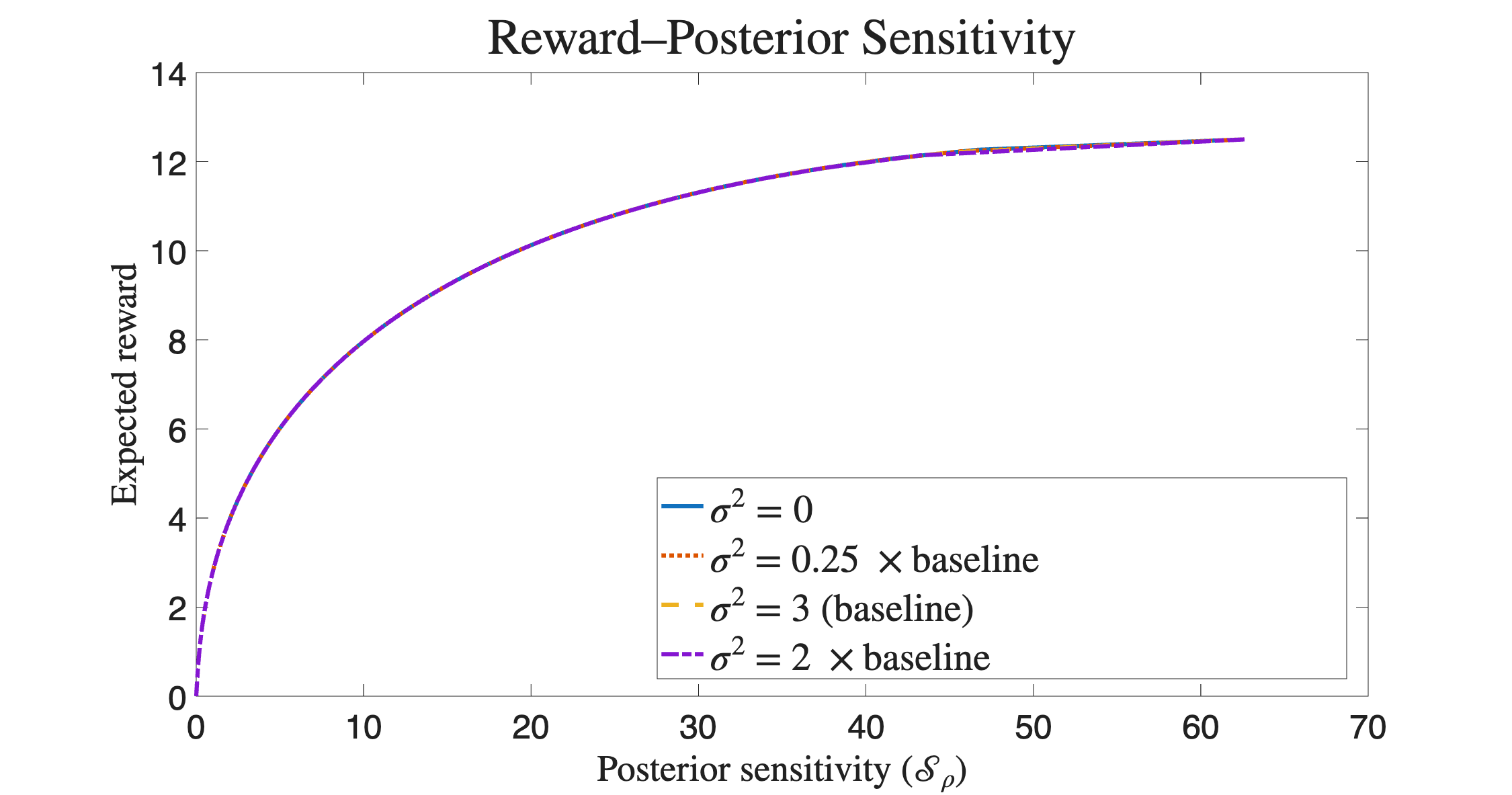} 
\\ [5pt]
\includegraphics[scale=0.25]{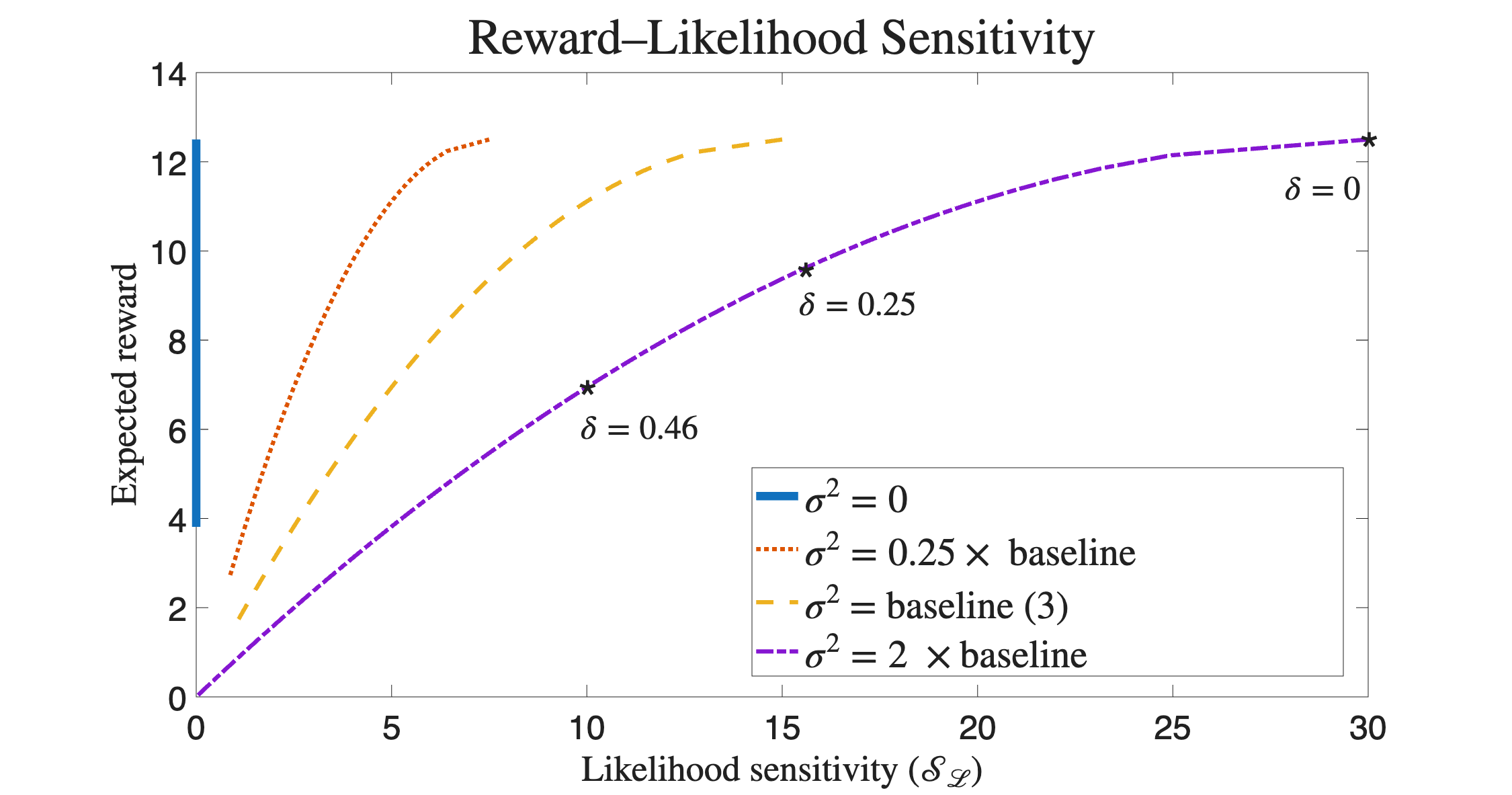} 
\end{center}
\caption{ The upper figure shows expected reward versus posterior sensitivity whereas the lower figure shows expected reward versus likelihood sensitivity for a collection of likelihood variances $\sigma^2$. Uncertainty sets for posterior and likelihood are equal in size (i.e., $\delta = \varepsilon$). In this example, the reward–posterior sensitivity frontier is independent of the likelihood variance, while the reward–likelihood sensitivity frontier changes with $\sigma^2$.
}
\label{fig:likelihoodvar-frontier}
\end{figure}

Figure \ref{fig:likelihoodvar-frontier} shows performance--robustness frontiers when the posterior variance is held fixed and the variance $\sigma^2$ of the likelihood function is varied. The reward--likelihood sensitivity frontier is now increasing in $\sigma^2$ whereas the reward--posterior frontier remains unchanged.




\section{Conclusion}
\label{sec:conclusion}

Bayesian models account for parameter uncertainty by specifying a prior on the uncertain parameters. However, the expected cost can still be sensitive to deviations from the  likelihood and the posterior (obtained by updating the user-specified prior using Bayes' rule and data) which may result in poor out-of-sample performance. Worst-case posterior sensitivity and worst-case likelihood sensitivity measure the sensitivity of the expected cost to worst-case deviations from the posterior and the likelihood, respectively. They are measures of robustness and a decision maker, concerned about the fragility of her assumptions, may prefer a decision for which both robustness measures are small. Worst-case posterior sensitivity vanishes when the posterior concentrates to a single point, which can occur when parameter uncertainty is reduced from learning. This is not the case with likelihood sensitivity. This is clear in the case of smooth $\phi$-divergence, where posterior sensitivity is the posterior standard deviation of the conditional expectation of the cost whereas likelihood sensitivity is the posterior expected value of its conditional standard deviation.
A nearly Pareto-optimal tradeoff between performance (expected cost) and robustness (worst-case posterior and likelihood sensitivity) can be mapped out by solving a Distributionally Robust Bayesian Optimization problem. 


\paragraph{\bf Acknowledgments}
Jun-ya Gotoh is supported by the MEXT Grant-in-Aid 24K01113.
Michael Kim is supported  by the Natural Sciences and Engineering Research Council (NSERC) Discovery Grant RGPIN-2015-04019.
Andrew Lim is supported by the  Ministry of Education, Singapore, under its  2024 Academic Research Fund Tier 2 grant call (Award ref: MOE-T2EP20224-0018).

\newpage

\bibliographystyle{plainnat}
\bibliography{references}

\newpage

\appendix

\section{Penalty version of the Bayesian model}
\label{App:BayesPenalty}
To compute worst-case sensitivity for the penalty version in the case of smooth $\phi$-divergence, we proceed as before by defining worst-case deviations from the nominal distribution by solving a worst-case problem, and computing the sensitivity using the worst-case distribution. We account for the Bayesian structure by considering perturbations of the posterior and likelihood, which allows us to quantify the impact of these choices on sensitivity. 

Let $\eta$ be a distribution for $\theta$ on $\Theta$  and ${\rm d}( \eta | \rho)$ be a deviation measure of $\eta$ from $\rho$.

Let ${\mathcal L} = \{ L^\theta| \theta\in\Theta\}$ be the nominal likelihood function and ${\mathcal Q} = \{{\mathbb Q}^\theta\vert\theta\in\Theta\}$ be a family of probability distributions for $Y$ parameterized by $\theta\in\Theta$. For every $\theta$, ${\rm d}({\mathbb Q}^\theta| L^\theta)$ is the deviation of ${\mathbb Q}^\theta$ from $L^\theta$. 
We define the deviation of ${\mathcal Q}$ from $\mathcal L$ as the expected deviation of the probability measure ${\mathbb Q}^\theta$ from ${\mathbb P}^\theta$ under the posterior $\rho$:
\begin{eqnarray}
{\rm d}({\mathcal Q}| {\mathcal L}) =  {\mathbb E}_\mathbb \rho \big\{{\rm d}({\mathbb Q}^\theta| L^\theta)\big\}.
\label{eq:dQL}
\end{eqnarray}

We find worst-case perturbations from the nominal model $\rho$ and $L^\theta$ by solving the worst-case problem
\begin{align*}
\lefteqn{\max_{{\eta, \,{\mathcal Q}} } {\mathbb E}_\eta \Big[{\mathbb E}_{{\mathbb Q}^{\theta}} \big\{f(x, Y)\big\}\Big] - \frac{1}{\delta}{\rm d}({\mathcal Q}| {\mathcal L}) - \frac{1}{\varepsilon} {\rm d}(\eta | \rho)} \\
& = \max_\eta  {\mathbb E}_\eta\left[\max_{{{\mathbb Q}^\theta\in {\mathcal Q}} }{\mathbb E}_{{\mathbb Q}^{\theta}} \big\{f(x, Y)\big\} - \frac{1}{\delta}{\mathbb E}_{L^\theta}\left\{\phi\left(\frac{{\rm d}Q^\theta}{{\rm d}L^\theta}\right)\right\}\right] - \frac{1}{\varepsilon}{\mathbb E}_\rho \left[\phi\left(\frac{{\rm d}\eta}{{\rm d}\rho}\right)\right]  \\
& = \max_\eta  {\mathbb E}_\eta \big[\Psi(\theta)\big] - \frac{1}{\varepsilon}{\mathbb E}_\rho \left[\phi\left(\frac{{\rm d}\eta}{{\rm d}\rho}\right)\right],
\end{align*}
where
\begin{eqnarray*}
\Psi(\theta) =\frac{1}{\delta} \ln {\mathbb E}_{L^{\theta}} \Big[e^{\delta f(x, Y)}\Big]
\end{eqnarray*}
and $\varepsilon$ controls deviations of the alternative posterior $\eta$ from the nominal $\rho$, while $\delta>0$ controls deviations of ${\mathbb Q}^\theta$ from $L^\theta$.
The worst-case distribution is given by
\begin{eqnarray*}
{\mathbb Q}^{\varepsilon,\delta}({\rm d}\theta, {\rm d}y) = \eta^\varepsilon({\rm d}\theta) {\mathbb Q}^{\theta,\delta}({\rm d}y),
\end{eqnarray*}
where
\begin{align*}
{\mathbb Q}^{\theta, \delta}  & =   \argmax_{\mathbb Q}{\mathbb E}_{{\mathbb Q}} \big[f(x, Y) \big] - \frac{1}{\delta}{\mathbb E}_{L^\theta}\left[\phi\left(\frac{{\rm d}{\mathbb Q}}{{\rm d}L^\theta}\right)\right],\\
\eta^\varepsilon & =  \argmax_\eta  {\mathbb E}_\eta \big[\Psi(\theta)\big] - \frac{1}{\varepsilon}{\mathbb E}_\rho \left[\phi\left(\frac{{\rm d}\eta}{{\rm d}\rho}\right)\right].
\end{align*}
The following result characterizes the worst-case distribution when $\delta$ and $\eta$ are small.
\begin{proposition} \label{prop:wc-bayes}
The worst-case measure ${\mathbb Q}^{\varepsilon, \delta}({\rm d}\theta, {\rm d}y)$
where
\begin{align*}
\frac{{\rm d}{\mathbb Q}^{\theta, \delta}}{{\rm d}L^{\theta}}  
& = [\phi']^{-1}\Big(\delta\big( f(Y)  -{\mathbb E}_{L^{\theta}}[f(Y)]\big)\Big) \\ 
& = 1 + \frac{\delta}{\phi{''}(1)}\Big\{f(Y)-{\mathbb E}_{L^\theta}[f(Y)]\Big\} + o(\delta),\\
\frac{{\rm d}\eta^\varepsilon}{{\rm d}\rho}
& = [\phi']^{-1}\Big(\varepsilon\big( \Psi(\theta)  -{\mathbb E}_{\rho}[\Psi(\theta)]\big)\Big) \\
& = 1+ \frac{\varepsilon}{\phi''(1)}\Big(\Psi(\theta)-{\mathbb E}_\rho[\Psi(\theta)]\Big).
\end{align*}
\end{proposition} 
We can now compute worst-case sensitivity. The expected reward under the worst-case measure is
\begin{align*}
V(\varepsilon, \delta; f(x,\cdot))  = {\mathbb E}_{{\mathbb Q}^{\varepsilon, \delta}}\big[f(x, Y)\big] 
  = {\mathbb E}_{\eta^{\varepsilon}} \Big[{\mathbb E}_{{\mathbb Q}^{\theta,\delta}} \big\{f(x, Y)\big\}\Big].
\end{align*}
By Proposition \ref{prop:wc-bayes}, the inner expectation
\begin{align*}
{\mathbb E}_{{\mathbb Q}^{\theta,\delta}} \big[f(x, Y)\big] & = {\mathbb E}_{L^\theta}\Big[\frac{{\rm d}{\mathbb Q}^{\theta, \delta}}{{\rm d}L^{\theta}} f(x, Y)\Big] \\
& = {\mathbb E}_{L^\theta}\Big[\Big(1 + \frac{\delta}{\phi{''}(1)}\Big\{f(Y)-{\mathbb E}_{L^\theta}[f(Y)]\Big\} \Big)f(x, Y)\Big]+o(\delta) \\
& = {\mathbb E}_{L^\theta}[f(x, Y)] + \frac{\delta}{\phi{''}(1)} {\mathbb V}_{L^{\theta}}[f(x, Y)] + o(\delta).
\end{align*}
It now follows that 
\begin{align*}
\lefteqn{V(\varepsilon, \delta; f(x, \cdot))}\\ 
& = {\mathbb E}_\rho\Big[\frac{{\rm d}\eta^{\varepsilon}}{{\rm d}\rho}\Big({\mathbb E}_{L^\theta}[f(x, Y)] + \frac{\delta}{\phi{''}(1)} {\mathbb V}_{L^{\theta}}[f(x, Y)] \Big)\Big] + o(\delta) \\
& = {\mathbb E}_\rho \Big[{\mathbb E}_{L^\theta}\big\{f(x, Y)\big\}\Big]+ \frac{\delta}{\phi{''}(1)}  {\mathbb E}_{\rho} \Big[{\mathbb V}_{L^\theta}\big\{f(x, Y)\big\}\Big] + \frac{\varepsilon}{\phi{''}(1)} {\mathbb V}_{\rho} \Big[{\mathbb E}_{L^\theta}\big\{f(x, Y)\big\}\Big]  + O(\varepsilon \delta),
\end{align*}
and the worst-case sensitivity with respect to the posterior and the likelihood are given by
\begin{align}
{\mathcal S}_{\rho}(f) & = \lim_{\varepsilon\rightarrow 0}\frac{1}{\varepsilon} \left\{V(\varepsilon, 0; f(x, \cdot)) - {\mathbb E}_\rho \Big[{\mathbb E}_{L^\theta}\big\{f(x, Y)\big\}\Big]\right\}  \nonumber \\ 
& =  \frac{1}{\phi{''}(1)} {\mathbb V}_{\rho} \Big[{\mathbb E}_{L^\theta}\big\{f(x, Y)\big\}\Big]
\label{eq:WCSprior} 
\end{align}
and
\begin{align}
{\mathcal S}_{\mathcal L}(f) 
& = \lim_{\delta\rightarrow 0} \frac{1}{\delta} \left\{V(0, \delta; f(x, \cdot)) - {\mathbb E}_\rho \Big[{\mathbb E}_{L^\theta}\big\{f(x, Y)\big\}\Big]\right\}  \nonumber \\ 
& = \frac{1}{\phi{''}(1)} {\mathbb E}_{\rho} \Big[{\mathbb V}_{L^\theta}\big\{f(x, Y)\big\}\Big].
\label{eq:WCSlikelihood}
\end{align}

\end{document}